\documentclass[12pt]{article}

\usepackage[enableskew]{youngtab}
\usepackage[leqno]{amsmath}
\usepackage{amsfonts}
\usepackage{amssymb}
\usepackage{amsthm}
\usepackage{amssymb}
\usepackage[all]{xy}
\usepackage{graphicx}
\usepackage{bbm}

\theoremstyle{plain}
\newtheorem{theorem}{Theorem}[section]
\newtheorem{corollary}[theorem]{Corollary}
\newtheorem{lemma}[theorem]{Lemma}

\newtheorem{proposition}[theorem]{Proposition}

\theoremstyle{definition}

\newtheorem{example}[theorem]{Example}

\title{A Frattini Theory for Leibniz Algebras}

\author{Chelsie Batten\\
\small Department of Mathematics\\[-0.8ex]
\small North Carolina State University, North Carolina, USA\\
\small \texttt{cabatten@ncsu.edu}\\
\and
Lindsey Bosko-Dunbar\\
\small Dept of Natural Sciences \& Mathematics\\[-0.8ex]
\small West Liberty University, West Virginia, USA\\
\small \texttt{lindsey.bosko@westliberty.edu}\\
\and
Allison Hedges\\
\small Department of Mathematics\\[-0.8ex]
\small North Carolina State University, North Carolina, USA\\
\small \texttt{armcalis@ncsu.edu}\\
\and
J. T. Hird\\
\small Department of Mathematics\\[-0.8ex]
\small North Carolina State University, North Carolina, USA\\
\small \texttt{jthird@ncsu.edu}\\
\and
Kristen Stagg\\
\small Department of Mathematics\\[-0.8ex]
\small University of Texas at Tyler, Texas, USA\\
\small \texttt{klstagg@ncsu.edu}\\
\and
Ernest Stitzinger\\
\small Department of Mathematics\\[-0.8ex]
\small North Carolina State University, North Carolina, USA\\
\small \texttt{stitz@math.ncsu.edu}\\
}

\date{August 11, 2011}

\begin{document}
\maketitle

\pagebreak
\begin{abstract}
A Frattini theory for non-associative algebras was developed in \cite{towersfrat} and results for particular classes of algebras have appeared in various articles.  Especially plentiful are results on Lie algebras.  It is the purpose of this paper to extend some of the Lie algebra results to Leibniz algebras.  
\end{abstract}

\section{Introduction}\label{sec:intro}

Following Barnes \cite{barnesleib}, we will say that an algebra is Leibniz if it satisfies the identity $x(yz) = (xy)z + y(xz)$.  Hence, left multiplication is a derivation.  (Many authors use a definition requiring that right multiplication is a derivation instead of our choice.) 

We focus on Leibniz algebras in which the Frattini ideal is 0.  Recall that the Frattini ideal is the largest ideal that is contained in the Frattini subalgebra.  For Lie algebras, these concepts coincide when either the algebra is solvable \cite{barnesgh} or when the underlying field has characteristic zero.  To the best of our knowledge, the only known time when they are not equal 
is the three dimensional cross product algebra over a field of characteristic 2 and algebras which have this algebra as a direct summand.  We show in Example \ref{CounterEx} that the Frattini subalgebra of a solvable Leibniz algebra need not be an ideal.  We also show that in the characteristic 0 case, the Frattini subalgebra is always an ideal.  For algebra $A$, the Frattini ideal will be denoted by $\Phi (A)$ and the Frattini subalgebra by $F(A)$.  We show that many of the Lie algebra results from references \cite{stitz1} - \cite{towers07} carry over to the Leibniz case.  In particular, we investigate elementary Leibniz algebras, 
those algebras $A$ in which the Frattini ideal of every subalgebra of $A$ is 0.

The Frattini subalgebra for Leibniz algebras already appears in \cite{barnesleib}.  There, Barnes extends his Lie algebra result which yields that if $A/W$ is nilpotent, then $A$ is nilpotent when $W$ is an ideal contained in the Frattini subalgebra of the Leibniz algebra $A$.  

Let $A$ be a Leibniz algebra.  The following is a review of material found in \cite{barnesleib}.  Any product of $n$ copies of $a$ is 0 unless the product is left normed; hence, we define $a^{n} = a(a(...(aa)...))$.  Define $A^{1} = A$ and inductively define $A^{n} = AA^{n-1}$.  It is known that the product of any $n$ elements from $A$ is in $A^{n}$.  $A$ is nilpotent of class $c$ if $A^{c+1} = 0$ but $A^{c} \neq 0$.  Let $Z(A)$ be the center of $A$; that is, $Z(A) = \left\{x | xa = 0 = ax \mbox{ for all }a \in A\right\}$.  $Z(A)$ is an ideal in $A$.  If $A$ is nilpotent of class $c$, then $A^{c} \subset Z(A)$ and therefore, $Z(A) \neq 0$.  Define the upper central series of $A$ as usual, and it follows that $A$ is nilpotent if and only if the upper central series terminates at $A$.

If $W$ is a subalgebra of $A$, define the left centralizer of $W$ as $Z^{l}_{A}(W) = \left\{x \in A | xW = 0\right\}$ and the centralizer of $W$ as $Z_{A}(W) = \{x \in A | xW = Wx = 0\}$.

We will often use the following result, which is Lemma 1.9 in \cite{barnesleib}.

\begin{lemma}\label{Lem1.9}
If $B$ is a minimal ideal of $A$, then either $BA = 0$ or $ba = -ab$ for all $a\in A$ and all $b\in B$.
\end{lemma}

Cartan subalgebras are defined in Leibniz algebras using the same definition as in Lie algebras; they are nilpotent, self-normalizing subalgebras.  They are discussed in \cite{barnesleib} and \cite{stitz2}.  Further results on Cartan subalgebras of Leibniz algebras can be obtained by directly following the proofs in Lie algebras.  In particular, as in Lie algebras, solvable Leibniz algebras have Cartan subalgebras (see Theorem 3 in \cite{barnes2} for a proof of the Lie algebra case) and

\begin{lemma}\label{NestedCartan}
If $W$ is an ideal in a Leibniz algebra $A$, $U$ is a subalgebra of $A$ with $W \subseteq U$, $U/W$ is a Cartan subalgebra of $A/W$ and $H$ is a Cartan subalgebra of $U$, then $H$ is a Cartan subalgebra of $A$.
\end{lemma}

The proof of this lemma is the same as the Lie algebra proof of Lemma 4 in \cite{barnes2}.


\section{Leibniz Algebras with $\Phi(A) = 0$}

The sum of nilpotent ideals is nilpotent by Corollary 3 of \cite{jacobsonleib}.  Therefore, define the nilradical of $A$ to be the largest nilpotent ideal of $A$ and denote it by $Nil(A)$.

\begin{lemma}\label{min nilp}
If $B$ is a minimal ideal in a nilpotent Leibniz algebra $A$, then
\begin{enumerate}
\item $B \cap Z(A) \neq 0$.
\item $B \subseteq Z(A)$.
\end{enumerate}
\end{lemma}

\begin{proof}
Since $B$ is minimal, either $BA = 0$ or $ba = -ab$ for all $a\in A$ and $b\in B$.  Hence it is sufficient to work on the left of $B$.  Let $A_{j+1}B = A(A_j B)$.  If $AB = 0$, both results hold.  Suppose that $AB \neq 0$.  Then there is an integer $k$ such that $A_{k+1}B = 0$ and $A_k B \neq 0$.  Hence $0 \neq A_k B \subseteq B \cap Z(A)$, so 1 holds.  Since $Z(A)$ is an ideal and $B$ is minimal, 2 holds.
\end{proof}

\begin{proposition}
Let $A$ be a Leibniz algebra and $B$ be a minimal ideal of $A$.  Then $Nil(A) \subseteq Z_{A}(B)$.
\end{proposition}

\begin{proof}
Let $N = Nil(A)$.  $N \cap B$ is 0 or $B$.  If it is 0, then $NB = 0 = BN$ and $N \subseteq Z_{A}(B)$.  If $N \cap B = B$, then $B$ is contained in $N$, hence $B \cap Z(N) \neq 0$ by Lemma \ref{min nilp}.  Since $B$ is minimal, this implies $B\subseteq Z(N)$.  Hence $NB = 0 = BN$, so $N \subseteq Z_A(B)$.
\end{proof}

The socle of $A$, $Soc(A)$ is the union of all minimal ideals of $A$ and is the direct sum of some of these minimal ideals.  The abelian socle of $A$, $Asoc(A)$ is the union of all abelian minimal ideals of $A$ and is the direct sum of some of them.  The concepts have counterparts in both Lie algebras and group theory.  The following several results also have analogues in Lie algebras and groups.

\begin{proposition}\label{SocReln}
Let $A$ be a a Leibniz algebra.  Then $Asoc(A) \subseteq Nil(A) \subseteq Z_{A}(Soc(A))$.
\end{proposition}

\begin{proof}
The first inclusion is clear.  By definition, $Soc(A) = \Sigma B_{i}$ where the $B_{i}$ are minimal ideals of $A$.  By the last proposition, $Nil(A) \subseteq Z_{A}(B_{i})$ for each $i$.  Then, $Nil(A) \subseteq \bigcap_{i} Z_A(B_{i}) = Z_{A} (\sum_{i} B_{i}) = Z_{A}(Soc(A))$.
\end{proof}

Let $A$ be a Leibniz algebra with $\Phi(A) = 0$.  If $W$ is an abelian ideal of $A$, then there exists a subalgebra $V$ of $A$ such that $A$ is the semidirect sum of $U$ and $V$ by Lemma 7.2 of \cite{towersfrat}.

\begin{theorem}\label{trip=}
Let $A$ be a Leibniz algebra with $\Phi(A) = 0$.  Then $Asoc(A) = Nil(A) = Z_{A}(Soc(A))$.
\end{theorem}

\begin{proof}
By Proposition \ref{SocReln}, it is enough to show that $Z_{A}(Soc(A)) \subseteq Asoc(A)$.  Let $C = Z_{A}(Soc(A))$ and $D = Asoc(A)$.  Since $D$ is abelian, $A$ is the semidirect sum of $D$ and a subalgebra $K$.  Then $D + (K \cap C) = (D + K) \cap C = A \cap C = C$.  Let $E = K \cap C$.  Since $C$ is an ideal, $E$ is invariant under multiplication by $K$ and is annihilated by $D$.  Therefore, $E$ is an ideal in $A$.  Let $B$ be a minimal ideal of $A$ contained in $E$.  Then $B^{2} \subseteq C \cdot Soc(A) = 0$.  Hence, $B \subset D$, which implies that $B \subseteq D \cap K = 0$.  Thus $E = 0$, so $C \subseteq D$.
\end{proof}

\begin{corollary}
\makebox[306pt][s]{ Let $A$ be a Leibniz algebra.  Then $Nil(A)$ $/$ $\Phi(A) \cong$} \\ 
%
%
$Asoc(A/\Phi(A))$.
\end{corollary}

\begin{proof}
This follows from the last result and Theorem 5.5 of \cite{barnesleib}.
\end{proof}

\begin{theorem}\label{AsocComp}
Let $A$ be a Leibniz algebra with $\Phi(A) = 0$.  Then $A = Asoc(A) \dotplus V$ where $V$ is a Lie algebra which is isomorphic to a subalgebra of the derivation algebra of Soc(A).
\end{theorem}

\begin{proof}
$Asoc(A)$ is complemented in $A$ by a subalgebra $B$ by Lemma 7.2 of \cite{towersfrat}.  For $x \in A$ let $L_{x}$ and $R_{x}$ be left and right multiplication by $x$ on $Soc(A)$ and $L: A \rightarrow Der(Soc(A))$ be defined as $L(x) = L_{x}$.  $L$ is a homomorphism, $Im(L)$ is a Lie algebra of derivations of $Soc(A)$ and $ker(L) = Z_{A}^{l}(Soc(A))$.  $Soc(A)$ is the direct sum of minimal ideals of $A$.  Let $W$ be one of these minimal ideals.  When acting on $W$, either $R_{x} = 0$ for all $x \in A$ or $R_{x} = -L_{x}$ for all $x \in A$ by 
Lemma \ref{Lem1.9}.  Hence, $Z_{A}^{l}(Soc(A)) = Z_{A}(Soc(A)) = Asoc(A)$ and $Im(L) \cong A/Asoc(A) \cong V$.
\end{proof} 

\begin{corollary}
Let $A$ be a semisimple Leibniz algebra.  Then $A$ is a Lie algebra.
\end{corollary}

\begin{proof}
$\Phi(A)$ is nilpotent by Theorem 5.5 of \cite{barnesleib}; hence $\Phi(A) = 0$.  Furthermore, $Z_{A}(Soc(A)) = Asoc(A) = 0$.  Hence, $K$ in Theorem \ref{AsocComp} is equal to $A$.
\end{proof}

From this corollary, we immediately obtain the following extension of the classical Lie algebra result.

\begin{corollary}
Let $A$ be a semisimple Leibniz algebra over a field of characteristic zero.  Then $A$ is the direct sum of ideals that are simple Lie algebras.
\end{corollary}

\begin{corollary}
Let $A$ be a Leibniz algebra over a field of characteristic zero with $\Phi(A) = 0$.  Then $A=Asoc(A)\dotplus V$, where $V = S \oplus Z(V)$ and $S$ is a semisimple Lie algebra.
\end{corollary}


\begin{proof}
By Theorem \ref{AsocComp}, $A = Asoc(A) \dotplus V$ where $V$ is a Lie subalgebra of $A$ that is isomorphic to a subalgebra $D$ of the derivation algebra of $Soc(A)$ and $D$ consists of left multiplication by elements of $A$.  Each minimal ideal is a minimal left ideal since right multiplication by $x \in A$ restricted to the minimal ideal is either 0 for all $x \in A$ or is the negative of left multiplication by $x$ for all $x \in A$ by 
Lemma \ref{Lem1.9}.  Hence, $D$ acts completely reducibly on $Soc(A)$.  Therefore, $D$ has the desired form by Theorem 11, chapter 2 of \cite{jacobson}.
\end{proof}

\begin{theorem}
Let $A$ be a Leibniz algebra over a field of characteristic 0.  Then $F(A)$ is an ideal in $A$.
\end{theorem}

\begin{proof}
If $B$ is an ideal in $A$ and $B$ is contained in $F(A)$, then $F(A/B) = F(A) / B$ by Proposition 4.3 of \cite{towersfrat}.  Hence it is enough to show the result when $B = \Phi(A) = 0$.  In this case, we need only show that $F(A) = 0$.  Then $A$ can be written as in Theorem \ref{AsocComp} and we use the notation that is used there.  Since $V$ is a Lie algebra, $F(V) = \Phi(V)$.  Furthermore, $V$ is the direct sum of ideals, each of which is simple or one dimensional.  Clearly $\Phi(V) = 0$.  Hence $F(A) \subseteq Asoc(A)$.  $Asoc(A)$ is the direct sum of minimal ideals $B_i$ of $A$, each of which is abelian.  Each $B_i$ is complemented by the maximal subalgebra $C_i$ which is the direct sum of $V$ and the remaining $B_j$.  Then $\cap C_i \subseteq V$ and $F(A) \subseteq V \cap Asoc(A) = 0$.  Therefore $F(A) = 0$.
\end{proof}


\section{Elementary Leibniz Algebras}

A Leibniz algebra $A$ will be called $elementary$ if the Frattini ideal of every subalgebra of $A$ is 0.  The analogous concept has been studied in both group theory and Lie algebras.

\begin{proposition}\label{A2comp}
Let $A$ be a Leibniz algebra such that $A^{2}$ is nilpotent.  Then the following are equivalent:
	\begin{enumerate}
		\item $\Phi(A) = 0$.
		\item $Nil(A) = Soc(A)$, and $Nil(A)$ is complemented by a subalgebra $C$.
		\item $A^{2}$ is abelian, is a semisimple $A$-module, and is complemented by a subalgebra $D$.
	\end{enumerate}
Furthermore, the complements to $A^{2}$ are precisely the Cartan subalgebras of $A$.
\end{proposition}

\begin{proof}
Assume that 1 holds.  $\Phi(A) = 0$.  Then $Nil(A) = Asoc(A) = Z_{A}(Soc(A)))$.  Since $A$ is solvable, $Asoc(A) = Soc(A)$.  Hence $Nil(A)$ is abelian and is complemented by a subalgebra.  Hence 2 holds.

Assume that 2 holds.  Since $A$ is solvable, $Nil(A) = Soc(A) = Asoc(A)$ and, since $A^{2}$ is nilpotent, $A^{2} \subseteq Nil(A)$.  Hence, $A^{2}$ is abelian.  Since $A$ acts completely reducibly on $Asoc(A)$, $A$ acts completely reducibly on $A^{2}$, and $A^{2}$ is complemented in $Asoc(A)$ by an ideal $B$.  Now $B + C$ complements $A^{2}$, and 3 holds.

Assume that 3 holds.  Since $\Phi(A) \subseteq A^{2}$ always holds, there exists an $A$-invariant subspace $B$ which complements $\Phi(A)$ in $A^{2}$.  Then $B + D$ complements $\Phi(A)$ in $A$.  Hence $\Phi(A) = 0$.

Under the conditions of the proposition, $Nil(A) = A^{2} \oplus Z(A)$ and both $A^{2}$ and $Z(A)$ are the direct sums of minimal ideals of $A$.  If $B$ is one of these minimal ideals, then $B$ is central or $AB + BA = B$.

If $H$ is a Cartan subalgebra of $A$, then $A = H + A^{2}$.  $H \cap A^{2}$ is an ideal in $A$ since $A^{2}$ is an abelian ideal.  Hence $H \cap A^{2}$ is the direct sum of minimal ideals of $A$.  If $B$ is one of them, then $B$ is a minimal ideal of $H$ which is nilpotent.  Hence $B \subseteq Z(A) \cap A^{2} = 0$.  Thus, $H \cap A^{2} = 0$ and $H$ is a complement to $A^{2}$ in $A$.

Conversely, let E be a complement to $A^{2}$ in $A$.  Let $x \in N_{A}(E) \cap A^{2}$ where $N_{A}(E)$ is the normalizer of $E$ in $A$.  Then $xE, Ex \subseteq E \cap A^{2} = 0$.  Therefore, $x \in Z(A) \cap A^{2} = 0$.  Hence $N_{A}(E) = E$ and $E$ is a Cartan subalgebra of $A$.
\end{proof}

It is not always the case that if $\Phi(A) = 0$ that $\Phi(M) = 0$ when $M$ is a subalgebra of $A$, even if $A$ is solvable as the following example shows.

\begin{example}\label{HeisEx}
Let $F$ be a field of characteristic $p$ where $p$ is prime and $V$ be a vector space with basis ${e_{1}, ..., e_{p}}$.  Define:

$x(e_{j}) = e_{j+1}$ with subscripts mod $p$

$y(e_{j}) = (j + 1)e_{j - 1}$ with subscripts mod $p$

$z(e_{j}) = e_{j}$.

Then $[y, x] = z$ and the other commutators between $x, y,$ and $z$ are 0.

Let $H$ be the three-dimensional Lie algebra with basis ${x, y, z}$.  Let $L=V\dotplus H$ with $[h, v] = h(v)$.  $L$ is a solvale Lie algebra.  Let $K = (I + R_{e_{2}})H$.  $R_{e_{2}}$ is a derivation since $L$ is a Lie algebra; hence, $I + R_{e_{2}}$ is an automorphism, and $K$ is a subalgebra with basis ${x + 3e_{1}, y + e_{3}, z + e_{2}}$.  It is easily checked that $H$ and $K$ are maximal subalgebras with $H \cap K = 0$.  Hence $\Phi(L) = 0$ while $\Phi(H) \neq 0$.
\end{example}

The Frattini subalgebra of a solvable Lie algebra is an ideal \cite{barnesgh}.  This result does not carry over to Leibniz algebras.  We continue with the construction in Example \ref{HeisEx}.

\begin{example}\label{CounterEx}
Let $F$, $V$, and $H$ be as in Example \ref{HeisEx}, with left multiplication of $H$ on $V$ the same as in that example, but $VH = 0$, and let $A$ be the Leibniz algebra $A = V \dotplus H$.  Then $Z_A(V) = Z_A^l(V) = V$.  Since $VA = 0$, $V \subseteq Z^l(A)$, the left center of $A$.  If $xA = 0$, then $xV = 0$.  Hence $Vx = 0$ and $x \in Z_A(V)$.  Hence $x\in V$ and $V = Z^l(A)$, since $V$ is a minimal ideal in $A$.  By Lemma 5.12 of \cite{barnesleib}, $V$ is complemented by a unique subalgebra in $A$, which is $H$.  $H$ is a maximal subalgebra of $A$ and every other maximal subalgebra of $A$ contains $V$.  These subalgebras are precisely those of the form $V\dotplus M$ where $M$ is maximal in $H$.  Let $\Omega$ be the set of all maximal subalgebras in $H$.  Then $F(A) = H \cap [\cap_{M\in\Omega} (V\dotplus M)] = H \cap [ V \dotplus \cap_{M\in\Omega} M] = \cap_{M\in \Omega} M = F(H) = Fz$, which is not an ideal in $A$.
\end{example}

\begin{theorem}\label{A2nilp-Aelem}
Let $A$ be a Leibniz algebra with $\Phi(A) = 0$ and $A^{2}$ nilpotent.  Then $A$ is elementary.
\end{theorem}

\begin{proof}
Let $M$ be a subalgebra of $A$.  Suppose that $A^{2} \subseteq M$.  Since $A^{2}$ is nilpotent, $\Phi(A^{2})$ is equal to the derived algebra of $A^{2}$, and then $\Phi(A^{2}) \subseteq \Phi(A) = 0$ because the derived algebra of $A^{2}$ is an ideal in $A$ and by Lemma 4.1 of \cite{towersfrat}.  Thus, $A^{2}$ is abelian and $A = A^{2} \dotplus H$, where $H$ is a Cartan subalgebra of $A$ by Proposition \ref{A2comp}.  Then $M = M \cap A = M \cap (A^{2} \dotplus H) = A^{2} \dotplus (M \cap H)$.  $A$ acts completely reducibly on $A^{2}$ and since $A^{2}$ is abelian, $H$ acts completely reducibly on $A^{2}$ as well.  Since $H$ is abelian, $H \cap M$ acts completely reducibly on $A^{2}$ and then also on $M^{2}$.  Thus, $M$ acts completely reducibly on $M^{2}$.  Furthermore, $M = A^{2} \dotplus (M \cap H) = B \oplus M^{2} \dotplus (M \cap H)$ for some ideal $B$ in $M$.  Now $B \oplus (M \cap H)$ is a complementary subalgebra to $M^{2}$ in $M$.  Thus, part 3 of Proposition \ref{A2comp} holds, and $\Phi(M) = 0$.

Suppose that $M$ does not contain $A^{2}$.  Then $M + A^{2}$ falls in the preceding case.  Hence we may assume that $M + A^{2} = A$.  Since $A^{2}$ is abelian, $A^{2} \cap M$ is an ideal in $A$.  Then $M/(A^{2} \cap M)$  complements $A^{2}/(A^{2} \cap M)$ in $A/(A^{2} \cap M)$ and $M/(A^{2} \cap M)$ acts completely reducibly on $A^{2}/(A^{2} \cap M)$.  Hence $M/(A^{2} \cap M)$ is a Cartan subalgebra of $A/(A^{2} \cap M)$ by the last proposition.  Let $H$ be a Cartan subalgebra of $M$.  By Lemma \ref{NestedCartan}, $H$ is a Cartan subalgebra of $A$.  Therefore, $H$ is a complement of $A^{2}$ in $A$.  Furthermore $M = M \cap (H + A^{2}) = H + (A^{2} \cap M)$ since $H \subseteq M$.  Therefore, $H$ is a complement to $A^{2} \cap M$ in $M$.  $A = A^{2} + M$ acts completely reducibly on $A^{2}$, hence $M$ does also since $A^{2}$ is abelian.  Therefore $M$ acts completely reducibly on $M^{2}$ and $A^{2} \cap M = M^{2} \oplus (A^{2} \cap Z(M))$.  Since $Z(M) \subseteq H$, it follows that $Z(M) \cap A^{2} = 0$.  Hence $H$ is a complement to $M^{2}$ in $M$.  Now $M$ satisfies part 3 of Proposition \ref{A2comp} and $\Phi(M) = 0$.
\end{proof}

As in Lie algebras, there is a converse to Theorem \ref{A2nilp-Aelem}.  The Lie algebra result is shown in \cite{towers07}.  We show the extension to Leibniz algebras in Theorem \ref{Aelem-A2nilp}.  The following is a direct extension of a Lie algebra result and the proof is the same as in the Lie algebra case as shown in Proposition 2 and Theorem 2 in \cite{stitz2}.

\begin{theorem}\label{FratSubset}
Let $A$ be a Leibniz algebra such that $A^{2}$ is nilpotent.  For any subalgebra $M$ of $A$, $\Phi(M) \subseteq \Phi(A)$.
\end{theorem}

The following result is Lemma 7.1 of \cite{towersfrat}.

\begin{lemma}
Let $A$ be an algebra and $B$ be an ideal in $A$.  If $U$ is a subalgebra of $A$ which is minimal with respect to the property $A = B+U$, then $B \cap U \subseteq \Phi(U)$.
\end{lemma}

The next result is shown for Lie algebras as Lemma 2.3 of \cite{towerselem}, but the proof is valid for Leibniz algebras also.

\begin{lemma}\label{idealcomp}
If $A$ is an elementary Leibniz algebra and $B$ is an ideal in $A$, then there exists a subalgebra $C$ such that $A = B + C$ and $B \cap C = 0$.
\end{lemma}

\begin{theorem}\label{Aelem-A2nilp}
Let $A$ be a solvable, elementary Leibniz algebra over a perfect field, $K$.  Then $A^2$ is nilpotent.
\end{theorem}

\begin{proof}
Let $A$ be a minimal counterexample.  Since $\Phi(A) = 0$, $A = B + C$, where $B = Nil(A) = Asoc(A)$ and $C$ is a Lie subalgebra.  If $C$ is nilpotent, then it is abelian, $A^2 \subseteq B$ and we are done.  Suppose that $C$ is not nilpotent.  Let $M_1$ be a maximal subalgebra of $C$ containing $C^2$ and $M = B + M_1$.  Then $A^2 \subseteq M$.  Then $M^2$ is an ideal in $A$ and is nilpotent by induction.  Hence $M^2 \subseteq Nil(A) = B$.  Therefore ${M_1}^2 \subseteq B \cap C = 0$, so $M_1$ is abelian.  Hence $A = B + (M_1 + D)$, where $D$ is a one dimensional subalgebra with basis $x$.

We claim that $B$ is the unique minimal ideal of $A$.  Suppose that $B = B_1 + \cdots + B_t$, where each summand is a minimal ideal in $A$ and $t > 1$.  Let $S_j = B_j + C$ for each $j$.  ${S_j}^2$ is nilpotent by induction.  Since $B_j$ is a minimal ideal in $A$, 
%
%
${S_j}^2 = B_j + C^2$ or $C^2$.  Hence ${S_1}^2 \cup \cdots \cup {S_t}^2$ is a Lie set whose span is $A^2$ and left multiplication by each $s$ in the Lie set is nilpotent on $A^2$.  Hence $A^2$ is nilpotent by the theorem in \cite{jacobsonleib}, which is a contradiction.  Hence $B$ is the unique minimal ideal in $A$.

We claim that $M_1 = Nil(C) = Asoc(C)$ is the unique minimal ideal in $C$.  Suppose that $M_1 = B_1 + \cdots + B_t$, where each summand is a minimal ideal of $C$ and $t > 1$.  Let $C_j = B+B_j+D$.  By induction, ${C_j}^2$ is nilpotent and ${C_j}^2 = {B_j}B + DB + DB_j$, since $B$ is an abelian minimal ideal in $A$ and $B_j$ is an abelian minimal ideal in $C$.  $L_c$ is nilpotent on ${C_j}^2$ for each $c\in {C_j}^2$.  Therefore $L_c$ is nilpotent on ${C_j}^2+B$, since $L_c({C_j}^2+B) \subseteq {C_j}^2$.  Since $B_iB_j = 0$, $L_c$ is nilpotent on $A^2 = {C_1}^2 + \cdots {C_t}^2$ for all $c\in {C_j}^2$.  Since ${C_1}^2\cup \cdots \cup {C_t}^2$ is a Lie set, $A^2$ is nilpotent by \cite{jacobsonleib}, a contradiction.  Hence $t=1$ and $A = B+C = B+(M_1+D)$ where $B$ is a minimal ideal in $A$, $M_1$ is a minimal ideal in $C$, and $D$ is a one dimensional subalgebra with basis $x$.

Let $E$ be the algebraic closure of $K$.  We will show that $E \otimes A^2$ is nilpotent, hence that $A^2$ is nilpotent.  Since $Nil(C) = Asoc(C) = M_1$, $L_x$ acts completely reducibly on $M_1$.  For each $m\in M_1$, $L_m$ acts completely reducibly on $B$ since $B \subseteq Asoc(B+Km)$. %
%
%
Passing to $E\otimes A$, since $K$ is perfect, the extension of these left multiplications are diagonalizable on $E\otimes M_1$ and $E\otimes B$ respectively.  In the second case, the $L_m$ commute, so the $L_m$ are simultaneously diagonalizable on $E\otimes B$.  We will show that $(E\otimes A)^2$ is nilpotent under these conditions, and hence assume that $K$ is algebraically closed.  Hence there is a basis of eigenvectors, $m_1, \ldots m_t$ in $M_1$ for $L_x$.  Therefore $xm_j = c_j m_j$ for some $c_j \in K$.  If $c_j = 0$ for any $j$, then 0 is an eigenvalue of $L_x$ on $M_1$.  
Since $M_1$ is irreducible over $K$, $M_1$ is one dimensional and $C$ is two dimensional abelian, a contradiction to $M_1$ being the unique minimal ideal in $C$.  Thus $c_j \neq 0$ for all $j$.  Also, $B$ is the direct sum of 
root spaces $B_{\alpha_i} = \{ b\in B | mb = \alpha_i(m) b \mbox{ for all } m\in M_1\}$.

Let $b_i \in B_{\alpha_i}$.  Then $xb_i = {b_1}' + \cdots + {b_n}'$, where ${b_j}' \in B_{\alpha_j}$.  Then
\begin{eqnarray*}
x(m_jb_i) &=& (xm_j)b_i + m_j(xb_i) \\
x\alpha_i(m_j)b_i &=& c_jm_jb_i + m_j\left(\textstyle\sum_k {b_k}'\right) \\
\alpha_i(m_j)\left(\textstyle\sum_k {b_k}'\right) &=& c_j\alpha_i(m_j)b_i + \textstyle\sum_k \alpha_k(m_j){b_k}' \\
c_j\alpha_i(m_j)b_i &=& \textstyle\sum_k \left(\alpha_i(m_j) - \alpha_k(m_j)\right) {b_k}'.
\end{eqnarray*}
Since $c_j \neq 0$ for all $j$, it follows that $m_jb_i = \alpha_i(m_j)b_i = 0$ for all $i$ and $j$.  Therefore $M_1B = 0$.  Since $B$ is a minimal ideal in $A$, $BM_1 = 0$ and $M_1 \subseteq Z_A(Asoc(A)) = Asoc(A)$, a contradiction.  Hence no minimal counterexample exists and the result holds.
\end{proof}

A construction of solvable elementary Lie algebras is given in \cite{towers07}.  Let $A$ be a vector space and let $B$ be an abelian Lie subalgebra of $gl(A)$ which acts completely reducibly on $A$.  Forming the semidirect sum of $A$ and $B$ one obtains an elementary, solvable, almost algebraic Lie algebra $L$.  We can use this construction to obtain elementary, solvable, Leibniz algebras.  Since $A = Asoc(L)$, $A$ is the direct sum of minimal ideals of $L$.  Let $A = A_1 + A_2$, where each summand is an ideal in $L$.  Define the Leibniz algebra $L^*$ to be the vector space $A+B$ with the same multiplication as in $L$ except that $A_2B = 0$.  This algebra has $\Phi(L^*) \subseteq (L^*)^2 \subseteq A$ and $\Phi(L^*) \subseteq B$, hence $\Phi(L^*) = 0$ and therefore, $L^*$ is elementary.  Following \cite{towers07}, such algebras will be called of type I.

\begin{corollary}
Let $A$ be a solvable Leibniz algebra over a perfect field $K$.  The following are equivalent:
\begin{enumerate}
\item $A$ is elementary.
\item $\Phi(A) = 0$ and $A^2$ is nilpotent.
\item $\Phi(A) = 0$ and $A$ is metabelian.
\item $Asoc(A)$ is complemented by an abelian subalgebra of $A$.
\item $A = B+E$, where $B$ is abelian and $E$ is elementary of type I.
\end{enumerate}
\end{corollary}

\begin{proof}
1 implies 2 is Theorem \ref{Aelem-A2nilp}.
2 implies 3.  By Theorem \ref{trip=} and Lemma 7.2 of \cite{towersfrat}, $A = Nil(A) + B$.  Since $A^2$ is nilpotent, $A^2 \subseteq Nil(A)$.  Hence $B$ is abelian.
3 implies 4.  Using Theorem \ref{trip=} and Lemma 7.2 of \cite{towersfrat}, $A = Nil(A) + B$ and $A/Nil(A)$ is abelian.
4 implies 5.  Let $Asoc(A) = Z(A) + K$, where $K$ is an ideal of $A$.  Put $E = K+B$, where $B$ is as in 4.  Then $Z_A(K) \cap B \subseteq Z(E) \subseteq Z(A) \cap E = 0$.  Hence $B$ is isomorphic to a subalgebra of $gl(K)$.  Since $K$ is contained in $Asoc(A)$, $K$ is completely reducible as a $B$-module.  Hence $E$ is of type I.
5 implies 1.  As in Lie algebras, the direct sum of elementary Leibniz algebras is elementary, hence $A$ is elementary.
\end{proof}


\section{Classification of Elementary Leibniz \\ Algebras}
%
%

A Leibniz algebra, $A$, is elementary if the Frattini ideal of every subalgebra of $A$ is 0.  By Theorem \ref{FratSubset}, if the underlying field has characteristic 0, then $A$ is elementary if the Frattini subalgebra of every subalgebra of $A$ is 0.  The following two theorems have been shown for Lie algebras in \cite{towerselem} and \cite{towers07} respectively.

\begin{theorem}
Let $A$ be a Leibniz algebra over an algebraically closed field $K$ of characteristic 0.  Then $A$ is elementary if and only if either
\begin{enumerate}
\item $A$ is the direct sum of copies of $sl_2(K)$ or
\item $A$ has basis $e_i$, $f_j$, for $i = 1, \ldots m$ and $j = 1, \ldots n$, with $e_if_j = \lambda_{ij}f_j$, either $e_if_j = - f_je_i$ or $f_je_i = 0$, and all other products between basis elements equal to 0 or
\item $A$ is the direct sum of an algebra from 1 and an algebra from 2.
\end{enumerate}
\end{theorem}

\begin{proof}
If $A$ is semisimple, then $A$ is a Lie algebra and it is of type 1, which follows from \cite{towerselem}.

In general $A = B + (Z \oplus S)$, where $B = Asoc(A)$, $Z$ is abelian, and $S$ is as in 1.  Suppose that $S$ is not 0.  We claim that $BS = SB = 0$.  Since $S$ is as in part 1, let $T$ be one of the simple summands in $S$.  Consider the subalgebra $C = B+T$ and let $D$ be a minimal ideal of $C$ contained in $B$.  If $\dim D = p \geq 2$, then there exists $t\in T = sl_2(K)$ and $a_1, a_2, \ldots, a_p$, with $ta_i = a_{i+1}$ for $i = 1, \ldots p-1$ and $ta_p = 0$.  Either $Dt = 0$ or $td = -dt$ for all $d\in D$.  In either case $\Phi(E) = E^2 \neq 0$, where $E$ is the subalgebra generated by $t$ and the $a_i$.  This contradiction forces $p=1$ and $TD=0=DT$.  Hence $SB = 0 = BS$.  Thus $Z$ acts completely reducibly on $B$.  Again using that $DS = 0$ or $ds = -sd$ for all $d\in D$ and $s\in S$, $Z$ acts completely reducibly on the left of $B$.  Hence $L_z$ is semisimple for each $z\in Z$.  Since $Z$ is abelian, the left multiplications of $z\in Z$ are simultaneously diagonalizable on $B$, hence there are bases for $Z$ and $B$ such that the left multiplications hold as in 2 and then the right multiplications hold using 
Lemma \ref{Lem1.9}.

The converse is clear.
\end{proof}

We next obtain the result of the last theorem for Leibniz algebras over fields of characteristic $p > 3$.

\begin{lemma}\label{perfect}
Let $K$ be perfect.  Let $A$ be an elementary Leibniz algebra over $K$.  Then $A = Asoc(A) + (B+S)$, where $B$ is abelian, $S$ is semisimple, and $BS+SB \subseteq B$.
\end{lemma}

\begin{proof}
$A = Asoc(A) + C$ by Lemma \ref{idealcomp}.  Let $B = Rad(C)$.  By the same Lemma, $C = B+S$, where $S$ is semisimple.  Clearly $Rad(A) = Asoc(A) + B$.  Since $Rad(A)$ is solvable and elementary, $Rad(A)^2$ is nilpotent by Theorem \ref{Aelem-A2nilp}, hence $B^2 \subseteq Nil(A) \cap B = Asoc(A) \cap B = 0$.
\end{proof}

\begin{theorem}
Let $A$ be a Leibniz algebra over an algebraically closed field $K$ of characteristic $p > 3$.  Then $A$ is elementary if and only if either
\begin{enumerate}
\item $A$ is the direct sum of copies of $sl_2(K)$ or
\item $A$ has basis $e_i$, $f_j$, for $i = 1, \ldots m$ and $j = 1, \ldots n$, with $e_if_j = \lambda_{ij}f_j$, either $e_if_j = - f_je_i$ or $f_je_i = 0$, and all other products between basis elements equal to 0 or
\item $A$ is the direct sum of an algebra from 1 and an algebra from 2.
\end{enumerate}
\end{theorem}

\begin{proof}
If $A$ is solvable, then $A = Asoc(A) + C$, where $C$ is abelian by Lemma \ref{perfect}.  Let $B$ be a minimal ideal of $A$.  Then $B \subseteq Nil(B+Kc) = Asoc(B+Kc)$ and $L_c$ acts completely reducibly on $B$ for each $c\in C$, hence $L_c$ acts diagonally on $B$.  Since $C$ is abelian, the $L_c$ are simultaneously diagonalizable on $B$ and $A$ is as in 2 follows from 
Lemma \ref{Lem1.9}.

If $A$ is semisimple, then $A$ is Lie and is as in 1 by Theorem 3.2 of \cite{towersfrat}.

If $A$ is neither solvable nor semisimple, then $A = B + (C+S)$, where $B = Asoc(A) \neq 0$, $C$ is abelian, $S$ is semisimple, and $CS+SC \subseteq C$ by the above lemma.  Again $S$ is as in 1.  Let $D_i = B+S_i$, where $S_i$ is a simple summand of $S$.  Then $B = Nil(D_i) = Asoc(D_i)$ and $B$ is a completely reducible $D_i$ module.  If $V$ is an irreducible submodule of $B$, then we claim $\dim V = 1$.  Since $V$ is an irreducible left $D_i$ Lie module, there exists $e\in D_i$ such that $V$ is a cyclic $Ke$ module and $e$ acts nilpotently on $V$.  Now $M = V+Ke$ is a nilpotent subalgebra since either $Ve=0$ or $ve = -ev$ for all $v\in V$.  Hence $M$ is abelian since it is elementary.  Therefore $eV = 0$ and $\dim V = 1$.  Hence $BS = SB = 0$.

If $C = 0$, we are done.  If not, then $C+S$ is elementary and $Nil(C+S) = Asoc(C+S) = C$ and as in the last paragraph, $CS =SC = 0$.  Then $A = (B+C) + S$, where $B+C$ is as in 2 and so $A$ is as in 3.
\end{proof}

\end{document}